\documentclass{proc-l}
\usepackage{amssymb,verbatim,amscd,mathrsfs,stmaryrd,wasysym,latexsym, bbm}

\begin{document}

\newcommand{\V}{{\mathcal V}}      
\renewcommand{\O}{{\mathcal O}}
\newcommand{\LL}{\mathcal L}
\newcommand{\Ext}{\hbox{\rm Ext}}
\newcommand{\Hom}{\hbox{Hom}}
\newcommand{\Proj}{\hbox{Proj}}
\newcommand{\GrMod}{\hbox{GrMod}}
\newcommand{\grmod}{\hbox{gr-mod}}
\newcommand{\tors}{\hbox{tors}}
\newcommand{\rank}{\hbox{rank}}
\newcommand{\End}{\hbox{{\rm End}}}
\newcommand{\Der}{\hbox{Der}}
\newcommand{\GKdim}{\hbox{GKdim}}
\newcommand{\pdim}{\hbox{pdim}}
\newcommand{\gldim}{\hbox{gldim}}
\newcommand{\im}{\hbox{im}}
\renewcommand{\ker}{\hbox{ker}}

\newcommand{\lonto}{{\protect \longrightarrow\!\!\!\!\!\!\!\!\longrightarrow}}

\newcommand{\g}{\frak g}
\newcommand{\h}{\frak h}
\newcommand{\fp}{\frak p}
\newcommand{\fa}{\frak a}
\newcommand{\fb}{\frak b}
\newcommand{\fc}{\frak c}
\newcommand{\m}{{\mu}}
\newcommand{\gl}{{\frak g}{\frak l}}
\newcommand{\ssl}{{\frak s}{\frak l}}

\newcommand{\ds}{\displaystyle}
\newcommand{\s}{\sigma}
\renewcommand{\l}{\lambda}
\renewcommand{\a}{\alpha}
\renewcommand{\b}{\beta}
\newcommand{\G}{\Gamma}
\newcommand{\z}{\zeta}
\newcommand{\e}{\epsilon}
\renewcommand{\d}{\delta}
\newcommand{\p}{\rho}
\renewcommand{\t}{\tau}

\newcommand{\C}{{\Bbb C}}
\newcommand{\N}{{\Bbb N}}
\newcommand{\Z}{{\Bbb Z}}
\newcommand{\ZZ}{{\Bbb Z}}
\newcommand{\Q}{{\Bbb Q}}
\renewcommand{\k}{\mathbbm{k}}

\newcommand{\K}{{\mathcal K}}

\newcommand{\rowxy}{(x\ y)}
\newcommand{\colxy}{ \left({\begin{array}{c} x \\ y \end{array}}\right)}
\newcommand{\scolxy}{\left({\begin{smallmatrix} x \\ y
\end{smallmatrix}}\right)}

\renewcommand{\P}{{\Bbb P}}

\newcommand{\la}{\langle}
\newcommand{\ra}{\rangle}
\newcommand{\tensor}{\otimes}

\newtheorem{thm}{Theorem}[section]
\newtheorem{lemma}[thm]{Lemma}
\newtheorem{cor}[thm]{Corollary}
\newtheorem{prop}[thm]{Proposition}

\theoremstyle{definition}
\newtheorem{defn}[thm]{Definition}
\newtheorem{notn}[thm]{Notation}
\newtheorem{ex}[thm]{Example}
\newtheorem{rmk}[thm]{Remark}
\newtheorem{rmks}[thm]{Remarks}
\newtheorem{note}[thm]{Note}
\newtheorem{example}[thm]{Example}
\newtheorem{problem}[thm]{Problem}
\newtheorem{ques}[thm]{Question}
\newtheorem{thingy}[thm]{}

\newcommand{\onto}{{\protect \rightarrow\!\!\!\!\!\rightarrow}}
\newcommand{\donto}{\put(0,-2){$|$}\put(-1.3,-12){$\downarrow$}{\put(-1.3,-14.5) 

{$\downarrow$}}}

\newcounter{letter}
\renewcommand{\theletter}{\rom{(}\alph{letter}\rom{)}}

\newenvironment{lcase}{\begin{list}{~~~~\theletter} {\usecounter{letter}
\setlength{\labelwidth4ex}{\leftmargin6ex}}}{\end{list}}

\newcounter{rnum}
\renewcommand{\thernum}{\rom{(}\roman{rnum}\rom{)}}

\newenvironment{lnum}{\begin{list}{~~~~\thernum}{\usecounter{rnum}
\setlength{\labelwidth4ex}{\leftmargin6ex}}}{\end{list}}



\title{Some non-Koszul algebras from rational homotopy theory}

\author[  Conner, Goetz ]{ }
\maketitle

\begin{center}

\vskip-.2in Andrew Conner \\
\bigskip

Department of Mathematics and Computer Science\\
Saint Mary's College of California\\
Moraga, CA 94575\\
\bigskip

 Pete Goetz \\
\bigskip

Department of Mathematics\\ Humboldt State University\\
Arcata, California  95521
\\ \ \\

\end{center}

\setcounter{page}{1}

\thispagestyle{empty}

\vspace{0.2in}

\begin{abstract}
The McCool group, denoted $P\Sigma_n$, is the group of pure symmetric automorphisms of a free group of rank $n$. The cohomology algebra $H^*(P\Sigma_n, \Q)$ was determined by Jensen, McCammond and Meier. We prove that $H^*(P\Sigma_n, \Q)$ is a non-Koszul algebra for $n \geq 4$, which answers a question of Cohen and Pruidze. We also study the enveloping algebra of the  graded Lie algebra associated to the lower central series of $P\Sigma_n$, and prove that it has two natural decompositions as a smash product of algebras.
\end{abstract}

\bigskip


\section{Introduction}

The main purpose of this article is to give a computer-aided proof of a surprising failure of the Koszul property for a family of algebras of interest in rational homotopy theory. 

The study of formal spaces provides a natural setting for the question of whether $H^*(X,\Q)$ is a Koszul algebra for a given space $X$. A space $X$ is called \emph{formal} if its Sullivan model $A_{PL}(X)$ of polynomial differential forms is a formal graded-commutative differential graded algebra (cdga)\ - that is, if $A_{PL}(X)$ is quasi-isomorphic as a cdga to $H^*(A_{PL}(X))$ with trivial differential. The $\Q$-completion of a formal space $X$, and hence $X$'s rational homotopy theory, is completely determined by $H^*(X,\Q)$.  There are many examples of such spaces in topology, including compact K\"{a}hler manifolds and complements of complex hyperplane arrangements. The latter implies Eilenberg-Mac Lane spaces of type $K(G,1)$ where $G$ is a pure braid group are formal spaces. 

Sullivan and Morgan \cite{Sullivan, Morgan} showed that the fundamental group of a formal space is a \emph{1-formal group}. There exist non-formal spaces with 1-formal fundamental groups, but Sullivan and Morgan's theorem does have a partial converse. Papadima and Suciu \cite{PapSuc} proved that for a connected, finite-type CW complex $X$, if $H^*(X,\Q)$ is a Koszul algebra and $\pi_1(X)$ is a 1-formal group, then $X$ is a formal space. 

In \cite{BerPap}, Berceanu and Papadima prove certain motion groups, denoted $P\Sigma_n$, are 1-formal groups. 
The groups $P\Sigma_n$ have a long history. David Dahm, in his unpublished PhD thesis \cite{Dahm}, considered arrangements of $n$ unknotted, unlinked circles in 3-space and the corresponding group $P\Sigma_n$ of motions of the arrangement in which each circle ends at its original position.  In \cite{Goldsmith}, Goldsmith explained Dahm's work and realized $P\Sigma_n$ as a subgroup of the automorphism group of the free group of rank $n$; she also determined specific generators of $P\Sigma_n$.


In \cite{McCool}, McCool found a presentation of $P\Sigma_n$ in terms of relations of the generators found by Goldsmith. In the paper \cite{BrownLee}, Brownstein and Lee were interested in the representation theory of $P\Sigma_n$ and successfully computed the second integral cohomology $H^2(P\Sigma_n, \Z)$; they also conjectured a presentation for the entire cohomology algebra $H^*(P\Sigma_n, \Z)$. Over ten years later, Jensen, McCammond and Meier in \cite{JMM} proved the conjecture of Brownstein and Lee. In particular, they showed $H^*(P\Sigma_n,\Q)$ is a quadratic algebra, a necessary condition for Koszulity. 

Berceanu and Papadima  \cite{BerPap} also established 1-formality for the related group $P\Sigma_n^+$ called the \emph{upper triangular McCool group}.
Cohen and Pruidze \cite{CohPru} proved $H^*(P\Sigma_n^+,\Q)$ is a Koszul algebra, hence the Eilenberg-Mac Lane space of type $K(P\Sigma^+_n, 1)$ is a formal space. They asked in that paper whether $H^*(P\Sigma_n,\Q)$ is Koszul.
Our main theorem is the following.

\begin{thm}\label{mainIntro} For $n \geq 4$, the algebra $H^*(P\Sigma_n, \Q)$ is not Koszul.
\end{thm} 

This raises the possibility that $K(P\Sigma_n, 1)$ is not a formal space, but we do not know if this is the case. Papadima and Yuzvinsky showed in \cite{Pap-Yuz} that for $X$, a formal, connected, topological space with finite Betti numbers, the Koszul property of $H^*(X, \Q)$ is equivalent to $X$ being a rational $K[\pi, 1]$ space, i.e., the $\Q$-completion of $X$ is aspheric. The proof of Theorem \ref{mainIntro} is given in Section 3.
One should contrast Theorem \ref{mainIntro} with the fact that several families of algebras naturally related to $H^*(P\Sigma_n, \Q)$ are Koszul. Besides $H^*(P\Sigma_n^+,\Q)$, the cohomology of the pure braid group is well-known to be Koszul, as are the algebras studied in \cite{BEER}.

In Section 2 we show that the quadratic dual of $H^*(P \Sigma_n, k)$ is the enveloping algebra of a graded Lie algebra. We denote this enveloping algebra by $U(\g_n)$, as such, $U(\g_n)$ has a natural bialgebra structure.  In Section 4 we study the algebra structure of $U(\g_n)$ and prove the following.

\begin{thm}\label{mainDecomposition} The algebra $U(\g_n)$ naturally decomposes as a smash product of algebras in two different ways.
\end{thm} 

This theorem parallels results of \cite{Bardakov, CohPak, Pettet}.

\section{Definitions of $U(\g_n)$ and $U(\g_n)^!$}
\label{defs}

Let $F_n$ denote the free group on $\{x_1,\ldots,x_n\}$. For $n\ge 2$, the group $P\Sigma_n$ is the subgroup of $\text{Aut}(F_n)$ generated by the automorphisms
$$\alpha_{ij}(x_k) = \begin{cases} x_jx_ix_j^{-1} & \text{ if } k=i\\ x_k & \text{ if } k\neq i,\\ \end{cases}$$

for all $1 \leq i \ne j \leq n$.

Elements of $P\Sigma_n$ have been called \emph{pure symmetric automorphisms} and \emph{basis-conjugating automorphisms}. McCool proved that the following relations determine a presentation of $P\Sigma_n$ \cite{McCool}:
$$\begin{cases} 
[\alpha_{ij},\alpha_{ik}\alpha_{jk}] & i, j, k \text{ distinct}\\
[\alpha_{ij},\alpha_{kj}] & i, j, k \text{ distinct}\\
[\alpha_{ij},\alpha_{kl}] & i, j, k, l \text{ distinct.}\\
\end{cases}
$$

A presentation for the integral cohomology of $P\Sigma_n$ was conjectured by Brownstein and Lee in \cite{BrownLee} and proved by Jensen, McCammond, and Meier in \cite{JMM}. In this section and the next we work with coefficients in $\Q$, though all results hold over a field $\k$ of characteristic 0.

\begin{thm}[\cite{JMM}]
\label{hilbSeries}
Let $n\ge 2$ and $E$ be the exterior algebra over $\Q$ generated in degree 1 by elements $\alpha_{ij}$, $1\le i\neq j\le n$. Let $I\subset E$ be the homogeneous ideal generated by
$\alpha_{ij}\alpha_{ji}$ for all $i\neq j$ and $$\alpha_{kj}\alpha_{ji}-\alpha_{kj}\alpha_{ki}+\alpha_{ij}\alpha_{ki}$$ for distinct $i,j,k$. 
As graded algebras,  $H^*(P\Sigma_n,\Q)\cong E/I$. In particular, the Hilbert series of $H^*(P\Sigma_n,\Q)$ is $h(t)=(1+nt)^{n-1}$.
\end{thm}

The cohomology algebra $H^*(P\Sigma_n,\Q)$ is a quadratic algebra, so it is natural to consider its quadratic dual algebra. We briefly recall this standard construction. The book \cite{PP} is an excellent reference on the theory of quadratic algebras.

Let $\k$ be a field. A $\k$-algebra $A$ is called \emph{quadratic} if there exists a finite dimensional $\k$-vector space $V$ and a subspace $R\subset V\tensor V$ such that $A\cong T(V)/\la R\ra$, where $T(V)$ denotes the tensor algebra on $V$. The tensor algebra $T(V)$ is graded by tensor degree, and if $R\subset V\tensor V$, the factor algebra $T(V)/\la R\ra$ inherits the tensor grading. Quadratic algebras are therefore $\N$-graded, have finite dimensional graded components, and are \emph{connected}, that is, $A_0=\k$. The \emph{quadratic dual algebra} $A^!$ is defined to be $T(V^*)/\la R^{\perp}\ra$ where $V^*$ is the $\k$-linear dual of $V$ and $R^{\perp}\subset V^*\tensor V^*$ is the orthogonal complement of $R$ with respect to the natural pairing. 

Throughout the paper, if $A$ is a quadratic algebra, we use the term \emph{$A$-module} to mean a $\Z$-graded $A$-module $M$ such that $\dim_{\k} M_i<\infty$ for all $i\in\Z$. If $M$ is an $A$-module and $j\in\Z$, $M[j]$ denotes the shifted module with $M[j]_i = M_{i+j}$ for all $i\in\Z$.
A homomorphism of $A$-modules $\varphi:M\rightarrow N$ is assumed to preserve degrees: $\varphi(M_i)\subset N_i$ for all $i\in\Z$.
The graded Hom functor is $$\underline{\Hom}_A(M,N):=\bigoplus_{j\in\Z} \Hom_A^j(M,N),$$ where $\Hom_A^j(M,N)=\Hom_A(M,N[-j])$. The right derived functors of the graded Hom functor are the graded Ext functors $$\Ext^i_A(M,N)=\bigoplus_{j\in\Z} \Ext_A^{i,j}(M,N).$$ The \emph{Yoneda algebra} of a quadratic algebra $A$ is the $\N\times\N$-graded $\k$-vector space $\Ext_A(\k,\k)=\bigoplus_{i,j} \Ext_A^{i,j}(\k,\k)$ endowed with the Yoneda composition product, which preserves the bi-grading. A quadratic algebra $A$ is called \emph{Koszul} if 
$\Ext_A^{i,j}(\k,\k)=0$ whenever $i\neq j$, in which case $\Ext_A(\k,\k)\cong A^!$ as graded algebras. An algebra $A$ is Koszul if and only if $A^!$ is (see \cite{PP}). 

In this section, and throughout most of this paper, we study the algebra $(E/I)^!$. 

\begin{lemma}
\label{presentation}
Let $n$, $E$, and $I$ be as in Theorem \ref{hilbSeries}.  
The quadratic algebra $(E/I)^!$ is isomorphic to the quotient of the free algebra $\Q\la x_{ij}\ |\ 1\le i\neq j\le n\ra$ by the homogeneous ideal generated by
\begin{align*}
[x_{ij},x_{ik}+x_{jk}]\qquad  & i, j, k\text{ distinct }\\ 
[x_{ik},x_{jk}]\qquad & i, j, k\text{ distinct }\\   
[x_{ij},x_{kl}]\qquad  & i, j, k, l\text{ distinct }
\end{align*}
where $[a,b]=ab-ba$.
\end{lemma}

\begin{proof}
If $\alpha_{ij}^*$ denotes the graded dual of the generator $\alpha_{ij}$, the mapping $x_{ij}\mapsto \alpha_{ij}^*$ on the degree one generators is clearly surjective. There are $n^2(n-1)(n-2)/2$ linearly independent relations listed, and the image of each vanishes on all relations of $E/I$ (including the relations of the exterior algebra). The Hilbert series of Theorem \ref{hilbSeries} shows $\dim_{\k} (E/I)_2 = n^2(n-1)(n-2)/2$, from which the result follows.
\end{proof}

The similarity to the relations of $P\Sigma_n$ is not a coincidence. Lemma \ref{presentation} shows that $(E/I)^!$ is the universal enveloping algebra of a quadratic Lie algebra which we denote by $\g_n$. The Lie algebra $\g_n$ is the associated graded Lie algebra of the Mal'cev Lie algebra of $P\Sigma_n$ (see \cite{BerPap}). For more on the construction of $\g_n$, see Section 4. Henceforth, we denote the quadratic dual algebra by $U(\g_n)$ rather than $(E/I)^!$. 

As a consequence of Lemma \ref{presentation}, it is easy to see that $U(\g_2)$ and $U(\g_3)$ are Koszul algebras; in fact, after a suitable change of variables, they are PBW-algebras in the sense of Priddy \cite{Priddy}. 

\begin{prop}
\label{23case}
For $n=2$ and $n=3$, the algebras $E/I$ and $U(\g_n)$ are isomorphic to PBW-algebras. In particular, they are Koszul algebras.
\end{prop}

\begin{proof}
The quadratic dual of a PBW-algebra is also a PBW-algebra, so it suffices to show $U(\g_2)$ and $U(\g_3)$ are isomorphic to PBW algebras.
The algebra $U(\g_2)$ is a free algebra on two generators, which is (trivially) a PBW-algebra. After making the linear change of variables $X_1=x_{21}+x_{31}$, $X_{2}=x_{12}+x_{32}$, and $X_3=x_{13}+x_{23}$, the defining relations of $U(\g_3)$ become
\begin{align*}
&[x_{12},X_3], [x_{13},X_2], [x_{21},X_3],\\
&[X_3-x_{13},X_1], [X_1-x_{21},X_2], [X_2-x_{12},X_1],\\
&[x_{21},X_1], [x_{12},X_2], [x_{13},X_3]
\end{align*}
Ordering monomials in the free algebra in deg-lex order where $x_{12}>x_{13}>x_{21}>X_1>X_2>X_3$, there are no ambiguities among the high terms of the defining relations. Therefore this alternate presentation of $U(\g_3)$ has a quadratic Gr\"{o}bner basis. It follows that $U(\g_3)$ is isomorphic to a PBW algebra, hence it is Koszul. 
\end{proof}

The analogous change-of-variables for $U(\g_4)$ does not produce a set of PBW generators, but many Koszul algebras are not PBW. 
The fact that the commutators $[x_{ij},x_{kl}]$ do not appear as relations of $U(\g_n)$ until $n=4$ underlies the failure of the PBW property. And as we will show, $U(\g_n)$ is not Koszul for $n=4$.

\section{The factor algebras $U(\g_n/\h_n)$}
\label{Ugh}


In this section, we describe a useful reduction from $U(\g_n)$ to a factor algebra whose failure of Koszulity implies $U(\g_n)$ is also not Koszul. We use the reduction to determine that $U(\g_4)$ is not Koszul. The next lemma shows this is sufficient to conclude $U(\g_n)$ is not Koszul for $n\ge 4$.

\begin{lemma}
\label{splitting lemma}
For $n\ge 4$, $U(\g_{n})$ is a split quotient of $U(\g_{n+1})$. Thus $U(\g_n)$ is Koszul if $U(\g_{n+1})$ is Koszul.
\end{lemma}

\begin{proof}
Since every relation of $U(\g_n)$ is also a relation of $U(\g_{n+1})$, there is a natural homomorphism $i_n:U(\g_n)\rightarrow U(\g_{n+1})$ given by $i_n(x_{ij})=x_{ij}$. We also have a well-defined homomorphism $\pi_n:U(\g_n)\rightarrow U(\g_{n-1})$ given by $\pi_n(x_{ij})=x_{ij}$ if $i\neq n\neq j$ and $\pi_n(x_{in})=\pi_n(x_{nj})=0$. The composition $\pi_{n+1}i_n$ is obviously the identity on $U(\g_n)$, so $U(\g_n)$ is a split quotient of $U(\g_{n+1})$. The induced  map  $i_{n}^*:\Ext_{U(\g_{n+1})}(\Q,\Q)\rightarrow \Ext_{U(\g_n)}(\Q,\Q) $ is therefore a bi-graded surjection, and thus $\Ext^{i,j}_{U(\g_{n})}(\Q,\Q)=0$ whenever $\Ext^{i,j}_{U(\g_{n+1})}(\Q,\Q)=0$. 
\end{proof}

For each $1\le j\le n$, let $X_j = \sum_{i\neq j} x_{ij}$ denote the sum of the ``$j$-th column'' generators of the Lie algebra $\g_n$. The following is based on an observation of Graham Denham. 

\begin{prop}
The Lie subalgebra $\h_n=\la X_1,\ldots, X_n\ra$ generated by the $X_j$ is a Lie ideal of $\g_n$. Furthermore, $\h_n$ is a free Lie subalgebra of $\g_n$. 
\end{prop}

\begin{proof}
Let $1\le i,j,k \le n$. Assume $i\neq j$. 
When $i,j,k$ are distinct, we have
$$[x_{ij},X_k]=[x_{ij},x_{ik}+x_{jk}]+\sum_{i\neq\ell\neq j} [x_{ij},x_{\ell k}]=0$$
When $j=k$, we have
$$[x_{ik},X_k] = \sum_{j\neq i} [x_{ik},x_{jk}] = 0$$
Finally,
$$[x_{kj},X_k] = \left[X_j-\sum_{i\neq k} x_{ij},X_k\right] = [X_j,X_k]-\sum_{i\neq k} [x_{ij},X_k]=[X_j,X_k]$$
These calculations show that $\h_n$ is a Lie ideal of $\g_n$.
Since $[X_j,X_k]=[x_{kj},X_k]$ and since Lie monomials of the form $[x_{kj},x_{ik}]$ do not appear in any relation of $\g_n$, it follows from standard facts about Groebner-Shirshov bases for Lie algebras \cite{Bokut} that $\h_n$ is free. 
\end{proof}

Each of $U(\g_n)$, $U(\h_n)$, and $U(\g_n/\h_n)$ is a quadratic algebra, so their Yoneda Ext-algebras are bigraded (see Section \ref{defs}).
We denote $H^{i,j}(\g_n,\Q)=\Ext^{i,j}_{U(\g_n)}(\Q,\Q)$ and similarly define $H^{i,j}(\h_n,\Q)$ and $H^{i,j}(\g_n/\h_n,\Q)$. When only one grading component is listed, it is assumed to be the homological degree. 

\begin{prop}
\label{ss}
For every $n\ge 2$, the enveloping algebra $U(\g_n)$ is a Koszul algebra if and only if $U(\g_n/\h_n)$ is a Koszul algebra.
\end{prop}

\begin{proof}
The argument is standard. We apply the Hochschild-Serre spectral sequence 
$$E_2^{p,q}=H^p(\g_n/\h_n,H^q(\h_n,\Q))\Rightarrow H^{p+q}(\g_n,\Q)$$
and henceforth we suppress the coefficients.
This is a spectral sequence of graded right $H^*(\g_n/\h_n)$-modules,  and the spectral sequence differential preserves the internal grading (see, for example, Section 6 of \cite{CS}).

Since $\h_n$ is free, $H^0(\h_n)=\Q$, $H^1(\h_n)=H^{1,1}(\h_n)=\Q^n$ and $H^q(\h_n)=0$ for $q>1$. It follows that $\g_n/\h_n$ acts trivially on $H^q(\h_n)$, so there is a bigraded vector space isomorphism
$$E_2^{p,q}\cong H^q(\h_n)\tensor H^p(\g_n/\h_n)$$
which is compatible with the spectral sequence differential and the right $H^*(\g_n/\h_n)$-module structure.

The $E_2$-page has at most two nonzero rows, so the differentials $d_2^{p,1}:E_2^{p,1}\rightarrow E_2^{p+2,0}$ are the only potentially nonzero maps. But $d_2^{0,1}=0$ because $E_2^{0,1}\cong H^1(\h_n)$ is concentrated in internal degree 1 and $E_2^{2,0}\cong H^2(\g_n/\h_n)$ is concentrated in internal degrees $\ge$ 2. Since the differential preserves the module action, $d_2^{p,1}=0$ for all $p$. Thus $E_2=E_{\infty}$ and there results an exact sequence
$$0\rightarrow H^{p+1}(\g_n/\h_n)\rightarrow H^{p+1}(\g_n)\rightarrow H^p(\g_n/\h_n)\tensor H^1(\h_n)\rightarrow 0$$
 of right $H^*(\g_n/\h_n)$ modules for all $p\ge 0$. Since $H^1(\h_n) = H^{1,1}(\h_n)$, and $H^{p+1,j}(\g_n)=0$ for $j<p+1$, it follows that $H^{p+1,j}(\g_n)=0$ for all $j\neq p+1$ if and only if $H^{p+1,j}(\g_n/\h_n)=H^{p,j-1}(\g_n/\h_n)=0$ for all $j\neq p+1$.
\end{proof}

We now consider the case $n=4$ and, for the remainder of this section, we abbreviate $U(\g/\h)=U(\g_4/\h_4) $. Eliminating generators $x_{41}, x_{42}, x_{43}$, and $x_{34}$ of $U(\g_4)$, we obtain the following presentation for $U(\g/\h)$.

\begin{lemma}
The algebra $U(\g/\h)$ is isomorphic to the free algebra on generators $x_{12}, x_{13}, x_{14}, x_{21}, x_{23}, x_{24}, x_{31}, x_{32}$ modulo the ideal generated by
$$
\begin{matrix}
[x_{21}, x_{31}] & [x_{12}, x_{32}] & [x_{13}, x_{23}] & [x_{14}, x_{24}]\\
[x_{13}, x_{24}] & [x_{14}, x_{23}] & [x_{14}, x_{32}] & [x_{24}, x_{31}]\\
[x_{31}, x_{12}+x_{32}] & [x_{32}, x_{21}+x_{31}] & [x_{13}, x_{12}+x_{32}] & [x_{23}, x_{21}+x_{31}]\\
[x_{21}, x_{13}+x_{23}] & [x_{12}, x_{13}+x_{23}] & [x_{21}, x_{14}+x_{24}] & [x_{12}, x_{14}+x_{24}].\\ 
\end{matrix}
$$
\end{lemma}

%
%

Having passed to $U(\g/\h)$, the quadratic dual algebra also becomes smaller. The proof of the following lemma is a straightforward calculation analogous to that of Lemma \ref{presentation}. 

\begin{lemma}
The algebra $U(\g/\h)^!$ is isomorphic to the exterior algebra over $\Q$ with generators $\a_{12}, \a_{13}, \a_{14}, \a_{21}, \a_{23}, \a_{24}, \a_{31}, \a_{32}$ modulo the ideal generated by
$$\a_{12}\a_{21}\quad \a_{13}\a_{31}\quad \a_{23}\a_{32}\quad \a_{23}\a_{24}\quad \a_{13}\a_{14}\quad \a_{24}\a_{32}\quad \a_{14}\a_{31} $$
$$\a_{12}\a_{31}-\a_{21}\a_{32}+\a_{31}\a_{32}\quad \a_{13}\a_{21}+\a_{23}\a_{31}+\a_{21}\a_{23}$$
$$\a_{14}\a_{21}+\a_{21}\a_{24}\quad \a_{12}\a_{13}-\a_{12}\a_{23}+\a_{13}\a_{32}\quad
\a_{12}\a_{14}-\a_{12}\a_{24}$$
In particular, the Hilbert series of $U(\g/\h)^!$ is $h(t)=(1+4t)^2$.
\end{lemma}

If a quadratic algebra $A$ is Koszul, the Hilbert series of $A$ and its quadratic dual algebra $A^!$ satisfy the relation $h_A(t)h_{A^{!}}(-t)=1$ (see \cite{PP}). The converse is false - there exist non-Koszul quadratic algebras for which the Hilbert series relation holds (see, for example \cite{Piontkovski}). The fact that $U(\g/\h)^!_3=0$ is noteworthy since the Hilbert series relation \emph{does} imply Koszulity in this case (Corollary 2.2.4 of \cite{PP}).

Thus the algebra $U(\g/\h)$ is Koszul \emph{if and only if} its Hilbert series is $1/(1-4t)^2$. Indeed, this was our motivation for considering the quotient $U(\g/\h)$.  The first nine terms of the Maclaurin series of $1/(1-4t)^2$ are
$$1+8t+48t^2+256t^3+1280t^4+6144t^5+28672t^6+131072t^7+589824t^8+\cdots.$$
As discussed in more detail below, we used the computer algebra system {\tt bergman}\ \cite{BergProgram} (specifically the {\tt ncpbhgroebner} method) to compute a Gr\"{o}bner basis and the Hilbert series for $U(\g/\h)$ in degrees $\le 8$. We found the first 8 terms of the Hilbert series to be
$$1+8t+48t^2+256t^3+1280t^4+6144t^5+28672t^6+131072t^7+5898\mathbf{34}t^8+\cdots.$$
This computation demonstrates our result.

\begin{thm}
\label{mainThm}
The algebra $U(\g/\h)$ is not Koszul. Consequently, $U(\g_n)$ is not Koszul for $n\ge 4$.
\end{thm}

We first performed this calculation several years ago using the 32-bit Windows installation of {\tt bergman} on an IBM ThinkPad T410 series with a 2.4 GHz Intel Core i5 M520 processor and 3GB DDR3 RAM. The result has since been duplicated several times on various systems and architectures as well as under GAP \cite{GAP4}, both by the authors and independently by Graham Denham, Fr\'{e}d\'{e}ric Chapoton, B\'{e}r\'{e}nice Oger, and Jan-Erik Roos. We wish to thank them for their interest in verifying Theorem \ref{mainThm}.

Several additional remarks are in order. First, one can also use {\tt bergman} to show the Hilbert series of $U(\g_4)$ agrees with $1/(1-4t)^3$ until $t^8$, at which point the coefficients again differ by 10, consistent with the collapse of the spectral sequence in Proposition \ref{ss}. Since $U(\g_4)$ is much bigger than $U(\g/\h)$, this calculation takes much longer and consumes much more memory.  The calculation for $U(\g/\h)$ takes less than 12 hours on most machines we have tried. 

Second, the results above show the algebras $U(\g_4/\h_4)$ and $U(\g_4)$ are \emph{7-Koszul} but not Koszul, meaning $H^{i,j}(\g_4)=0$ for $i<j<8$. The delayed failure of Koszulity together with the exponential growth of $U(\g/\h)$ make this calculation extremely difficult to perform by hand. 

The Hilbert series calculation shows $\dim H^{3,8}(\g/\h)=10$. Using {\tt bergman}, we have also computed $\dim H^{3,9}(\g/\h)=40$. This raises the following question.

\begin{ques}
Is $H^*(\g_n)$ finitely generated?
\end{ques}



\section{$U(\g_n)$ as a smash product}

In this section we show that $U(\g_n)$ can be decomposed as a {\it smash product} of algebras. We briefly recall the smash product construction following \cite{BrownBook}. Let $\k$ denote a field. Let $\g$ be a $\k$-Lie algebra acting on a $\k$-algebra $R$ by $\k$-derivations. For $x \in \g$ and $r \in R$, let $x(r)$ denote the action of $x$ on $r$. Let $\{x_i \ | \ i \in I\}$ be a basis for $\g$. The smash product $R \# U(\g)$ is defined to be the $\k$-algebra generated by $R$ and $\{x_i \ | \ i \in I\}$ with the relations $x_ir - rx_i = x_i(r)$ and $x_ix_j - x_jx_i = [x_i, x_j]$. As a vector space, $R \# U(\g) \cong R \tensor U(\g)$.

Let $G$ be a group. The lower central series of $G$ is the sequence of normal subgroups of $G$ defined by $G_1 = G$, $G_{n+1} = [G_n, G]$. Let $G(n) = G_n/G_{n+1}$ denote the $n$-th lower central series quotient group. Then $\g = \bigoplus_n G(n)\tensor \Q$ is a Lie algebra over $\Q$ where the Lie bracket is induced by the group commutator. In the case $G=P\Sigma_n$, the resulting Lie algebra is $\g_n$, see Theorem 5.4 (2) in \cite{BerPap}.

Let $1\rightarrow A\rightarrow B\rightarrow C\rightarrow 1$ be a split short exact sequence of groups. Assume $[A,C]\subset [A,A]$ or, equivalently, that $C$ acts trivially on $H_1(A)$. Then $0\rightarrow A(n)\rightarrow B(n)\rightarrow C(n)\rightarrow 0$ is a split exact sequence of abelian groups for all $n$, \cite{Falk, Kohno}. 
And since $\Q$ is a flat $\Z$-module, the sequence
$$0\rightarrow \mathfrak a\rightarrow \mathfrak b\rightarrow \mathfrak c\rightarrow 0$$
is a split exact sequence of graded $\Q$-Lie algebras. It then follows (see 1.7.11 in \cite{BrownBook}) that $U(\mathfrak b)\cong U(\mathfrak a)\#U(\mathfrak c)$. 

Bardakov, \cite{Bardakov}, has given two semidirect product decompositions of $P\Sigma_n$. Recall that $P\Sigma_n$ has a set of generators $\{\a_{ij} \ | \ 1 \leq i \ne j \leq n \}$. Suppose that $n \geq 3$, and set $m = \lfloor \frac{n}{2} \rfloor$. We define subgroups of $P\Sigma_n$ as follows. Let $K_n$ be the subgroup generated by $\{\a_{in}, \a_{ni} \ | \ 1 \leq i < n\}$; let $H_n$ be the subgroup generated by $\{\a_{2i-1, 2i}, \a_{2i, 2i-1} \ | \ i = 1, 2, \ldots, m \}$; let $G_n$ be the subgroup generated by the complement of $\{\a_{2i-1, 2i}, \a_{2i, 2i-1} \ | \ i = 1, 2, \ldots, m \}$ in the set of generators of $P\Sigma_n$. Then there are split short exact sequences of groups
$$1 \to K_n \to P\Sigma_n \to P\Sigma_{n-1} \to 1,$$
$$1 \to G_n \to P\Sigma_n \to H_n \to 1.$$

It is proved in \cite{CohPak} that the natural conjugation action of $P\Sigma_{n-1}$ on $H_1(K_n)$ is trivial. Similarly the formulas in the proof of Theorem 2 of \cite{Bardakov} show that the action of $H_n$ on $H_1(G_n)$ is also trivial. Let $R_n$, $S_n$, $T_n$ denote the enveloping algebras of the graded $\Q$-Lie algebras associated to the lower central series of the groups $G_n$, $H_n$, $K_n$, respectively. The  following theorem is immediate.

\begin{thm}
 The algebra $U(\g_n)$ decomposes as a smash product in two different ways.
 \begin{enumerate}
 \item $U(\g_n) \cong T_n \# U(\g_{n-1})$
 \item $U(\g_n) \cong R_n \# S_n$
 \end{enumerate}
\end{thm}

It is worthwhile to note that \cite{Bardakov} shows that $H_n$ is isomorphic to the direct product of $m$ copies of $F_2$, the free group on two generators. Hence $S_n$ is a tensor product of $m$ copies of a free associative algebra on two generators.

We conclude by proving that $T_n$ and $R_n$ are not finitely presented algebras.

\begin{lemma}
\label{splittings}
 Let $A_n$ denote either $T_n$ or $R_n$. There exists a split surjection $A_{n+1} \to A_n$. Consequently, there is a graded surjection $\Ext^i_{A_{n+1}}(\Q, \Q) \to \Ext^i_{A_{n}}(\Q, \Q)$.
\end{lemma}

\begin{proof} We prove the statement for the algebra $T_n$. We define $p: U(\g_{n+1}) \to U(\g_{n})$ by $$p(x_{ij}) = 0 \text{ for } i, j \ne n+1; \quad p(x_{n \, n+1}) = p(x_{n+1\,  n}) = 0;$$ $$p(x_{i \, n+1}) = x_{i n}, \quad p(x_{n+1\, j}) = x_{n j}, \text{ for } 1 \leq i, j < n.$$ It is easy to see that these assignments at the level of tensor algebras send the defining relations of $U(\g_{n+1})$ to zero, thus we have a well-defined algebra homomorphism. Restricting $p$ to the subalgebra $T_{n+1}$, we have a surjection $p:T_{n+1} \onto T_n$.

To split the map $p$, first note that the symmetric group $\Sigma_{n+1}$ acts on the algebra $U(\g_{n+1})$ by algebra automorphisms $$\s.x_{ij} = x_{\s(i) \, \s(j)} \text{ for all } \s \in S_{n+1}, \ x_{ij} \in U(\g_{n+1}).$$ Let $\tau$ denote the automorphism of $U(\g_{n+1})$ induced by the transposition $(n \ n+1)$. Let $i: U(\g_{n}) \to U(\g_{n+1})$ be the canonical inclusion, $i(x_{ij}) = x_{ij}$. Then the map $\t i$ is a splitting of $p$.

The last statement of the lemma follows as in the proof of Lemma \ref{splitting lemma}. The proof for the algebra $R_n$ is similar and straightforward.
\end{proof}

\begin{lemma} Let $\k$ be a field. Let $A$ be a connected, graded $\k$-algebra. Suppose that $n = \gldim(A) < \infty$. Let $M$ be a nonzero, finite dimensional, graded $A$-module. Then $\Ext^n_A(\k, M) \ne 0$.
\end{lemma}

\begin{proof}
We induct on $\dim M$. If $\dim M = 1$, then $M$ is a 1-dimensional trivial module, and $\Ext^n_A(\k, M) \ne 0$ since $\pdim(\k) = n$. Now suppose $\dim M = k > 1$. If $M$ is concentrated in a single degree, then $M$ is a direct sum of trivial modules and the result follows. So suppose $M_l \ne 0$, $M_j = 0$ for all $j > l$, and $M_l \ne M$. Then $S = M_l$ is a trivial submodule of $M$. Consider the short exact sequence $0 \to S \to M \to M/S \to 0$ and the piece of the associated long exact sequence $$\Ext^n_A(\k, M) \to \Ext^n_A(\k, M/S) \to 0.$$
By induction $\Ext^n_A(\k, M/S) \ne 0$, so it follows that $\Ext^n_A(\k, M) \ne 0$, as desired.
\end{proof}

\begin{prop} Let $\k$ be a field. Suppose $0\rightarrow \mathfrak a\rightarrow \mathfrak b\rightarrow \mathfrak c\rightarrow 0$ is a split exact sequence of graded $\k$-Lie algebras, $\gldim(U(\fc)) = 1$ and $g_b = \gldim(U(\fb)) < \infty$. If $M = \Ext^{g_b}_{U(\fa)}(\k, \k) \ne 0$, then $M$ is infinite dimensional over $\k$. 
\end{prop}

\begin{proof} We use the Cartan-Eilenberg spectral sequence (\cite{CE}) $$E^{p,q}_2 = \Ext^p_{U(\fc)}(\k, \Ext^q_{U(\fa)}(\k, \k)) \implies \Ext^{p+q}_{U(\fb)}(\k, \k).$$ Since $\gldim(U(\fc)) = 1$, the spectral sequence collapses on the $E_2$-page. Thus $E_2^{1, g_b} = E_{\infty}^{1, g_b} \hookrightarrow \Ext^{g_b+1}_{U(\fb)}(\k, \k)$. Hence $E_2^{1, g_b} = 0$. The result follows from the previous lemma.
\end{proof}

\begin{thm} Let $n \geq 3$. The normal subalgebras $T_n$ and $R_n$ of $U(\g_n)$ are not finitely presented. 
\end{thm}

\begin{proof} We note that in the Bardakov decompositions of $P\Sigma_3$, $K_3 = G_3$; also $H_3 = U(\g_2)$ is a free algebra on two generators. We note that $\gldim(U(\g_3)) = 2$ (which follows from Proposition \ref{23case}), and $\Ext^2_{T_3}(\Q,\Q) \ne 0$ (since $[x_{13}, x_{23}]$ is a defining relation of $T_3$). Now applying the last proposition we know $\Ext^2_{T_3}(\Q,\Q)$ is infinite dimensional. Therefore Lemma \ref{splittings} implies that $\Ext^2_{T_n}(\Q,\Q)$ and $\Ext^2_{R_n}(\Q,\Q)$ are infinite dimensional for all $n \geq 3$. Hence $T_n$ and $R_n$ are not finitely presented. 

\end{proof}

This result should be compared with Pettet's results in \cite{Pettet}.

Although these algebras are not finitely presented it is straightforward to compute minimal defining relations in any given degree. We have checked that the {\it quadratic closures} (the algebras on the same sets of generators but with only the defining quadratic relations) of $T_n$ and $R_4$ are Koszul. Further study of these algebras would be interesting.

\bibliographystyle{acm}
\bibliography{bibliog2}

\end{document}